\documentclass[a4paper,12pt]{article}
\usepackage{latexsym,amssymb,amsmath,amsthm,amsfonts,euscript}

\begin{document}

\newcounter{lemma}
\newcommand{\lemma}{\par \refstepcounter{lemma}%
{\bf Lemma \arabic{lemma}.}}

\newcounter{theorem}
\newcommand{\theorem}{\par \refstepcounter{theorem}%
{\bf Theorem \arabic{theorem}.}}



\newcounter{corolary}
\newcommand{\corolary}{\par \refstepcounter{corolary}%
{\bf Corollary \arabic{corolary}.}}


\title{Estimates of the modulus of continuity of the logarithmic double layer potential in the closure of domain}
\author{Sergiy Plaksa\thanks{Institute of Mathematics of the National Academy of Sciences of Ukraine, Kyiv, Ukraine.
E-mail: \texttt{plaksa62@gmail.com}} \and Alexander Sarana\thanks{Zhytomyr Ivan Franko State University, Zhytomyr, Ukraine.
E-mail: \texttt{sarana.alexandr@gmail.com}}}
\maketitle

\begin{abstract}
\noindent
We obtain estimates of the modulus of continuity for the real part of the Cauchy-type integral in the closure of domain bounded by an Ahlfors-regular integration curve.
These estimates are more exact than the well-known Zygmund estimate for the modulus of continuity
of the Cauchy-type integral.
The accuracy of estimates is proved by constructing an example of a curve and an integral density for which the specified estimates are exact with respect to the order of smallness.
\end{abstract}


\medskip

\section{\large Introduction}
Let $\gamma$ be a closed rectifiable Jordan curve in the complex plane $\mathbb{C}$, and
let $D^{+}$ and $D^{-}$ be the interior and exterior domains
bounded by $\gamma$, respectively.

The classical theory of the logarithmic double layer potential (see, for example, J.~Plemelj \cite{Plemelj-1911})
\begin{equation}\label{log-pot}
\frac{1}{2\pi}\int\limits_{\gamma} g(t)\frac{\partial}{\partial {\bf n}_t}\left(\ln\frac{1}{|t-z|}\right)\,ds_t
 \qquad \forall\,z\in D^{\pm}
\end{equation}
is developed in the case where the integration curve $\gamma$ is a Lyapunov curve, and the integral density $g : \gamma\rightarrow\mathbb{R}$ taking values in the set of real numbers $\mathbb{R}$ is a continuous function.
Here ${\bf n}_t$ and $s_t$ denote the unit vector of the outward normal to the curve $\gamma$ at a point $t\in\gamma$ and an arc coordinate of this point, respectively.

J.~Radon \cite{Radon-46} established the continuous extension of the logarithmic double layer potential from the domains $D^{+}$ and $D^{-}$ to the boundary $\gamma$ for any continuous integral density $g : \gamma\rightarrow\mathbb{R}$
in the case where $\gamma$ is a curve of bounded rotation, i.e., a curve for which the angle between the tangent to the curve and a fixed direction is a function of bounded variation.
It is known that the class of Lyapunov curves and the class of Radon curves of bounded rotation are different, i.e., each of them contains curves that do not belong to the other class (see, for example, I.I.~Danilyuk \cite[p.~26]{Dan-m}).

J.~Kr\'al \cite{Kral-1-1964} proved that logarithmic double layer potential \eqref{log-pot} is extended continuously from the domains
$D^{+}$ and $D^{-}$ to the boundary $\gamma$ for all continuous functions $g : \gamma\rightarrow\mathbb{R}$
if and only if the curve $\gamma$ satisfies the condition
 \begin{equation}\label{um-Kral}
\sup_{\xi\in\gamma}\,\int\limits_0^{2\pi} \mu_{\gamma}(\xi,\phi)\,d\phi<\infty,
\end{equation}
where $\mu_{\gamma}(\xi,\phi)$  is the number of intersection points of the curve $\gamma$ with the ray $\{z=\xi+re^{i\phi} : r>0\}$.

Logarithmic double layer potential (\ref{log-pot}) is the real part of the Cauchy-type integral
(see, for example, I.I.~Danilyuk \cite{Dan-m}, F.D.~Gakhov \cite{Gah}, N.I.~Muskhelishvili \cite{Mus})
\begin{equation}\label{C-type-int}
\widetilde{g}(z):=
\frac{1}{2\pi i}\int\limits_{\gamma}\frac{g(t)}{t-z}\,dt
 \qquad \forall\,z\in D^{\pm}\,.
\end{equation}

While the results mentioned above about the continuous extension of the logarithmic double layer potential
to the boundary of a domain, which are contained in papers \cite{Plemelj-1911,Radon-46,Kral-1-1964}, are valid for arbitrary continuous functions $g$, the continuous extension of integral (\ref{C-type-int}) to the boundary of $D^{+}$ or $D^{-}$ requires additional assumptions on
the density $g$ of the integral.

The theory of the boundary properties of the integral in (\ref{C-type-int})  is presented in the monographs by F.D.~Gakhov \cite{Gah} and N.I.~Muskhelishvili \cite{Mus}  under the classical assumptions about the smoothness of the integration curve and the H\"older density of the integral.
In the papers of N.A.~Davydov \cite{Davydov}, V.V.~Salaev \cite{Salaev},
T.S.~Salimov \cite{Salimov}, E.M.~Dyn'kin \cite{Dynkin}, O.F.~Gerus \cite{G77,G96-2-4}, the theory of the Cauchy-type integral and the Cauchy singular integral is developed on an arbitrary rectifiable Jordan curve
in classes that are more general than the H\"older class of the integral density,
which are defined, as a rule, in terms of the modulus of continuity of the function $g$.
In this case, for the modulus of continuity of the Cauchy-type integral and its limiting values on the boundary of a domain,
both upper estimates (see \cite{Salaev,G77,Salimov,Dynkin,G96-2-4}) and lower estimates for specially constructed functions to prove the ordinal accuracy of the specified upper estimates on subclasses of the class of closed rectifiable Jordan curves (see, for example, V.V.~Salaev \cite{Salaev} and O.F.~Gerus \cite{G82,G99}) are established.

The aim of this work is to make more exact the mentioned estimates for the modulus of continuity 
of the real part of the Cauchy-type integral (i.e., the logarithmic double layer potential) on subclasses of the class of closed rectifiable Jordan curves.

\vskip 2mm

\section{An upper estimate for the modulus of continuity of the real part of the Cauchy-type integral in the case of an Ahlfors-regular integration curve}
In what follows, a closed rectifiable Jordan curve $\gamma$ satisfies the condition (see V.V.~Salaev \cite{Salaev})
\begin{equation}\label{2-4:nerivnist'-teta}
\theta(\varepsilon):=\sup\limits_{x\in\gamma} \
\theta_{x}(\varepsilon)=O(\varepsilon),\quad
\varepsilon\rightarrow 0\,,
\end{equation}
where $\theta_{x}(\varepsilon):={\rm mes}\,\gamma_{\varepsilon}(x)$,\,\,
$\gamma_{\varepsilon}(x):=\{t\in\gamma :
|t-x|\le \varepsilon\}$ and\, ${\rm mes}$\, denotes the linear Lebesgue measure on $\gamma$.
Curves satisfying condition (\ref{2-4:nerivnist'-teta}) are important in solving various problems
(see, for example, V.V.~Salaev \cite{Salaev}, L.~Ahlfors \cite{Ahlfors2-2-4}, G.~David \cite{David-2-4},
C.~Pommerenke \cite{Pomeren-6-2}, A.~B\"ottcher and Y.I.~Karlovich \cite{Bot-Karl-2-4}).
Such curves are often called {\it regular} (see, for example, \cite{David-2-4}) or {\it Ahlfors-regular}
(see, for example, \cite{Pomeren-6-2}), or {\it Carleson curves} (see, for example, \cite{Bot-Karl-2-4}).

Let $d:=\max\limits_{t_1,t_2\in\gamma} |t_1-t_2|$ be the diameter of the curve $\gamma$.

For each $\xi\in\gamma$,
we define a branch $\arg(z-\xi)$ continuous on $\gamma\setminus\{\xi\}$
of the multivalued function ${\rm Arg}\,(z-\xi)$ in the following way.
For each positive $\delta<d/2$, we select that connected component $\gamma_{\xi,\delta}$ of the set $\gamma_{\delta}(\xi)$ which contains the point $\xi$, and we take such a point $\xi_1\in\gamma_{\xi,\delta}$ at which there is a tangent to $\gamma$ and which does not precede the point $\xi$ under the given orientation of the curve $\gamma$. It is obvious that in the case in which there is a tangent to $\gamma$ at the point $\xi$, we can set $\xi_1=\xi$.
Let us cut the complex plane along the curve $\Gamma_{\xi,\delta}:=\gamma[\xi,\xi_1]\cup\Gamma[\xi_1,\infty]$, where $\gamma[\xi,\xi_1]$ is the arc of $\gamma$ with the initial point $\xi$ and the end point $\xi_1$, and $\Gamma[\xi_1,\infty]$ is a smooth curve that connects the points $\xi_1$ and $\infty$ and lies completely (except for its ends $\xi_1$ and $\infty$) in the domain $D^-$.
Now, let us single out a branch $\arg_{\delta}(z-\xi)$ of the multivalued function ${\rm Arg}\,(z-\xi)$, which is continuous outside the cut
$\Gamma_{\xi,\delta}$ with the normalization condition $\arg_{\delta}(z_0-\xi)=\phi_0$, where $z_0\in D^+$ and $\phi_0$ is one of the values of the function ${\rm Arg}\,(z-\xi)$ at $z=z_0$. We shall use the fixed values $z_0$ and $\phi_0$ for all positive $\delta<d/2$.
As a result, we have the obvious equality
\[\arg_{\delta_1}(t-\xi)=\arg_{\delta}(t-\xi) \qquad \forall\,\delta_1,\delta : 0<\delta_1<\delta<d/2 \quad \forall\,t\in\gamma\setminus\gamma_{\delta}(\xi)\]
that implies the existence of the following limit:
\[\arg(t-\xi):=\lim\limits_{\delta\to 0^+}\,\arg_{\delta}(t-\xi) \qquad \forall\,t\in\gamma\setminus\{\xi\}\,.\]

By definition
\begin{equation}\label{Re-siC}
\int\limits_{\gamma}
 \big(g(t)-g(\xi)\big)\,d\arg(t-\xi):=\lim\limits_{\delta\to 0^+}\,\int\limits_{\gamma\setminus\gamma_{\delta}(\xi)}
 \big(g(t)-g(\xi)\big)\,d\arg(t-\xi) \qquad \forall\,\xi\in\gamma\,,
  \end{equation}
  if the limit on the right-hand side of the equality exists.

Denote
\[M_{\gamma}(g,\varepsilon):=  \sup\limits_{\xi\in\gamma}\,\sup\limits_{\delta\in(0,\varepsilon)}\,\bigg|\,\int\limits_{\gamma_{\varepsilon}(\xi)\setminus\gamma_{\delta}(\xi)}
 \big(g(t)-g(\xi)\big)\,d\arg(t-\xi)\,\bigg|\,, \qquad \varepsilon>0\,.
 \]

It is proved in Theorem 2 \cite{Plaksa-EMJ} that the condition
\begin{equation}\label{nidu-nepr}
M_{\gamma}(g,\varepsilon)\to 0\,, \qquad \varepsilon\to 0\,,
\end{equation}
is necessary and sufficient for the continuous extension to the boundary $\gamma$ of the real part of Cauchy-type integral (\ref{C-type-int}) from the domains $D^+$ and $D^-$ in the case where $\gamma$ is a closed Ahlfors-regular Jordan curve.
Moreover, if condition \eqref{nidu-nepr} is satisfied, then the limit on the right-hand side of equality \eqref{Re-siC} exists, and for all $\xi\in\gamma$, the limiting values
\begin{equation}\label{gran-znach-double-pot}
\big({\rm Re}\,\widetilde{g}\big)^{\pm}(\xi):=\lim\limits_{z\to\xi,\,z\in D^{\pm}}\,{\rm Re}\,\widetilde{g}(z)
\end{equation}
 are represented by the formulas
\begin{equation}\label{double-pot+}
 \big({\rm Re}\,\widetilde{g}\big)^{+}(\xi)=g(\xi)+\frac{1}{2\pi}\,\int\limits_{\gamma}
 \big(g(t)-g(\xi)\big)\,d\arg(t-\xi)\,,
\end{equation}
\begin{equation}\label{double-pot-}
 \big({\rm Re}\,\widetilde{g}\big)^{-}(\xi)=\frac{1}{2\pi}\,\int\limits_{\gamma}
 \big(g(t)-g(\xi)\big)\,d\arg(t-\xi)\,.
\end{equation}

Let us note that in the paper of O.F.~Gerus and M.~Shapiro \cite{Ger-Sh-1}, an analog of the Davydov theorem \cite{Davydov} is proved for an appropriate Cauchy-type integral along an arbitrary rectifiable Jordan curve in $\mathbb{R}^2$, which takes values in the algebra of quaternions.
This result is applied in the paper of O.F.~Gerus and M.~Shapiro \cite{Ger-Sh-2} to establish sufficient conditions for the continuous extension to the boundary of a domain of metaharmonic potentials, a partial case of which is logarithmic double layer potential \eqref{log-pot}.

Continuous extensions of the function ${\rm Re}\,\widetilde{g}$ to the closures  $\overline{D^{+}}$ and $\overline{D^{-}}$ will be denoted by $\big({\rm Re}\,\widetilde{g}\big)^{+}$ and $\big({\rm Re}\,\widetilde{g}\big)^{-}$, respectively.

For a function $f \colon E\rightarrow \mathbb{C}$ continuous on $E\subset \mathbb{C}$, we shall use its modulus of continuity
\[\omega_{E}(f, \varepsilon):=\sup\limits_{t_1, t_2 \in E\,:\, |t_1-t_2|\le \varepsilon}
\left|f(t_1)-f(t_2)\right|. \]

Consider auxiliaries statements.

\vskip 1mm

\begin{lemma}\label{lem-1}
{\em Let a closed Jordan curve $\gamma$ be Ahlfors-regular. Let a function
$g \colon \gamma\rightarrow \mathbb{R}$ be continuous on $\gamma$ and let
condition \eqref{nidu-nepr} be satisfied.
Let $\xi_1, \xi_2\in\gamma$, $|\xi_1-\xi_2|\le\varepsilon_1$ and $\varepsilon/2\le\varepsilon_1\le 5\varepsilon$ for $\varepsilon\in(0,d/8]$.
Then the following estimate is true:
\[ \bigg|\frac{1}{2\pi}\int\limits_{\gamma_{2\varepsilon_1}(\xi_1)} \big(g(t)-g(\xi_2)\big)\,d\arg(t-\xi_2)\bigg|
\le c\,\left(\,M_{\gamma}(g,\varepsilon)+ \varepsilon\,\int\limits_{\varepsilon}^{15\varepsilon}\frac{\omega_{\gamma}(g,\eta)}{\eta^2}\,d\eta \right),\]
where the constant\, $c$\, depends only on\, $\gamma$\, but does not depend on $\varepsilon$.}
\end{lemma}

\vskip 1mm

\textit{\textbf{Proof.}}
Let us evaluate the modulus of each term on the right-hand side of the equality
\begin{multline*}
\int\limits_{\gamma_{2\varepsilon_1}(\xi_1)} \big(g(t)-g(\xi_2)\big)\,d\arg(t-\xi_2)=\\
=\int\limits_{\gamma_{\varepsilon}(\xi_2)} \big(g(t)-g(\xi_2)\big)\,d\arg(t-\xi_2)+
{\rm Re}\,\bigg(\frac{1}{i}\, \int\limits_{\gamma_{2\varepsilon_1}(\xi_1)\setminus\gamma_{\varepsilon}(\xi_2)}\frac{g(t)-g(\xi_2)}{t-\xi_2}\,dt\bigg)
=: J_1+J_2.
\end{multline*}

It is obvious that
\[ |J_1| \le M_{\gamma}(g,\varepsilon)\,.\]

Taking into account the two-sided inequality $\varepsilon/2\le\varepsilon_1\le 5\varepsilon$ and condition (\ref{2-4:nerivnist'-teta}), we obtain a chain of inequalities
\begin{multline*}
 |J_2|\le
\int\limits_{\gamma_{2\varepsilon_1}(\xi_1)\setminus\gamma_{\varepsilon}(\xi_2)}\frac{|g(t)-g(\xi_1)|+|g(\xi_1)-g(\xi_2)|}{|t-\xi_2|}\,|dt|
\le  \frac{2\omega_{\gamma}(g,2\varepsilon_1)}{\varepsilon}\,\theta_{\xi_1}(2\varepsilon_1)\le \\
\le c\,\omega_{\gamma}(g,2\varepsilon_1)\le c\,\omega_{\gamma}(g,2\varepsilon_1)\,\varepsilon_1\,\int\limits_{2\varepsilon_1}^{3\varepsilon_1}\frac{d\eta}{\eta^2}\le
    c\,\varepsilon\,\int\limits_{\varepsilon}^{15\varepsilon}\frac{\omega_{\gamma}(g,\eta)}{\eta^2}\,d\eta\,,
\end{multline*}
where\, $c$\, denotes different constants whose values depend only on the curve\, $\gamma$.
\hfill $\Box$

\vskip 2mm

\begin{lemma}\label{lem-2}
{\em Under the conditions of Lemma \ref{lem-1} the following estimate is true:
\begin{equation}\label{Zygmund-log-pot-boundary}
\bigg|\big({\rm Re}\,\widetilde{g}\big)^{\pm}(\xi_1)-\big({\rm Re}\,\widetilde{g}\big)^{\pm}(\xi_2)\bigg|
\le c\,\left(\,M_{\gamma}(g,\varepsilon)+ \varepsilon\,\int\limits_{\varepsilon}^{2d}\frac{\omega_{\gamma}(g,\eta)}{\eta^2}\,d\eta \right),
\end{equation}
where the constant\, $c$\, depends only on the curve\, $\gamma$.}
\end{lemma}

\vskip 1mm

\textit{\textbf{Proof.}}
Let us prove estimate (\ref{Zygmund-log-pot-boundary}) for the function
$\big({\rm Re}\,\widetilde{g}\big)^{+}$ (estimate (\ref{Zygmund-log-pot-boundary}) for the function $\big({\rm Re}\,\widetilde{g}\big)^{-}$ is similarly proved).

We take into account equality (\ref{double-pot+}) and use the following representation of the difference:
\begin{multline*}
\big({\rm Re}\,\widetilde{g}\big)^{+}(\xi_1)-\big({\rm Re}\,\widetilde{g}\big)^{+}(\xi_2)=g(\xi_1)-g(\xi_2)+\\[2mm]
+\frac{1}{2\pi}\,\int\limits_{\gamma}
 \big(g(t)-g(\xi_1)\big)\,d\arg(t-\xi_1)-\frac{1}{2\pi}\,\int\limits_{\gamma}
 \big(g(t)-g(\xi_2)\big)\,d\arg(t-\xi_2)=\\
=g(\xi_1)-g(\xi_2)+\frac{1}{2\pi}\,\int\limits_{\gamma_{2\varepsilon_1}(\xi_1)}\big(g(t)-g(\xi_1)\big)\,d\arg(t-\xi_1)-
\frac{1}{2\pi}\,\int\limits_{\gamma_{2\varepsilon_1}(\xi_1)}\big(g(t)-g(\xi_2)\big)\,d\arg(t-\xi_2)+\\
+{\rm Re}\,\bigg(\frac{\xi_1-\xi_2}{2\pi i}\, \int\limits_{\gamma\setminus\gamma_{2\varepsilon_1}(\xi_1)}\frac{g(t)-g(\xi_1)}{(t-\xi_1)(t-\xi_2)}\,dt\bigg)+{\rm Re}\,\bigg(\frac{g(\xi_2)-g(\xi_1)}{2\pi i}\, \int\limits_{\gamma\setminus\gamma_{2\varepsilon_1}(\xi_1)}\frac{dt}{t-\xi_2}\bigg)=:\\[2mm]
 =:g(\xi_1)-g(\xi_2)+I_1-I_2+I_3+I_4\,.
\end{multline*}

The moduli of integrals $I_1$ and $I_2$ are estimated in Lemma \ref{lem-1}.

To estimate the modulus of the term $I_3$, we use the inequality $|t-\xi_2|\ge |t-\xi_1|/2$ for all $t\in\gamma\setminus\gamma_{2\varepsilon_1}(\xi_1)$,
Proposition 7.2 \cite{Pla-Shpak-mono} (see also the proof of Theorem 1 in the paper of O.F.~Gerus \cite{G78}) and condition
(\ref{2-4:nerivnist'-teta}) so that we have
\begin{multline*}
|I_3|\le \frac{|\xi_1-\xi_2|}{\pi}\,\int\limits_{\gamma\setminus\gamma_{2\varepsilon_1}(\xi_1)}\frac{|g(t)-g(\xi_1)|}{|t-\xi_1|^2}\,|dt|\le
\frac{\varepsilon_1}{\pi}\int\limits_{[2\varepsilon_1,d]}\frac{\omega_{\gamma}(g,\eta)}{\eta^2}\,d\theta_{\xi_1}(\eta)\le\\
\leq \frac{2\varepsilon_1}{3\pi}\,\int\limits_{\varepsilon_1}^{d}\frac{\theta_{\xi_1}(2\eta)\omega_{\gamma}(g,2\eta)}{\eta^3}\,d\eta\leq
c\,\varepsilon\int\limits_{\varepsilon}^{2d}\frac{\omega_{\gamma}(g,\eta)}{\eta^2}\,d\eta\,,
\end{multline*}
where the constant\, $c$\, depends only on the curve $\gamma$.

In addition, taking into account the inequality (see the proof of Theorem 1 in the paper of V.V.~Salaev \cite{Salaev})
\[ \bigg|\, \int\limits_{\gamma\setminus\gamma_{2\varepsilon_1}(\xi_1)}\frac{dt}{t-\xi_2}\,\bigg|\le 4\pi\,,\]
we obtain the following relations:
\begin{multline*}
\big|g(\xi_1)-g(\xi_2)+I_4\big|\le \big|g(\xi_2)-g(\xi_1)\big|+\frac{\big|g(\xi_2)-g(\xi_1)\big|}{2\pi}\,\bigg|\, \int\limits_{\gamma\setminus\gamma_{2\varepsilon_1}(\xi_1)}\frac{dt}{t-\xi_2}\,\bigg|\le
3\,\omega_{\gamma}(g,\varepsilon)=\\
= 6\,\omega_{\gamma}(g,\varepsilon) \,\varepsilon\,\int\limits_{\varepsilon}^{2\varepsilon}\frac{d\eta}{\eta^2}\le
    6\,\varepsilon\,\int\limits_{\varepsilon}^{2d}\frac{\omega_{\gamma}(g,\eta)}{\eta^2}\,d\eta\,.
\end{multline*}
\hfill $\Box$

\vskip 1mm

\begin{lemma}\label{lem-3}
{\em Let a closed Jordan curve $\gamma$ be Ahlfors-regular. Let a function
$g \colon \gamma\rightarrow \mathbb{R}$ be continuous on $\gamma$ and let
condition \eqref{nidu-nepr} be satisfied.
Let $z\in D^{\pm}$ and let $\xi_z$ be one of the closest points of the curve $\gamma$ to the point $z$, and let
$|z-\xi_z|\le 2\varepsilon$ with $\varepsilon\in(0,d/8]$.
Then the following estimate is true:
\[ \bigg|\big({\rm Re}\,\widetilde{g}\big)^{\pm}(z)-\big({\rm Re}\,\widetilde{g}\big)^{\pm}(\xi_z)\bigg|
\le c\,\left(\,M_{\gamma}(g,\varepsilon)+ \varepsilon\,\int\limits_{\varepsilon}^{2d}\frac{\omega_{\gamma}(g,\eta)}{\eta^2}\,d\eta \right),\]
where the constant\, $c$\, depends only on the curve\, $\gamma$.}
\end{lemma}

\vskip 1mm

\textit{\textbf{Proof.}}
Denote $\rho:=|z-\xi_z|$.

Taking into account the Cauchy integral theorem, the Cauchy integral formula, equalities (\ref{double-pot+}) and (\ref{double-pot-}), we have
 \begin{multline*}
 \big({\rm Re}\,\widetilde{g}\big)^{\pm}(z)-\big({\rm Re}\,\widetilde{g}\big)^{\pm}(\xi_z)=
 {\rm Re}\,\bigg(\frac{1}{2\pi i}\,\int\limits_{\gamma}\frac{g(t)-g(\xi_z)}{t-z}\,dt\bigg)-\frac{1}{2\pi}\,\int\limits_{\gamma}
 \big(g(t)-g(\xi_z)\big)\,d\arg(t-\xi_z)=\\
 ={\rm Re}\,\bigg(\frac{1}{2\pi i}\,\int\limits_{\gamma_{2\rho}(\xi_z)}\frac{g(t)-g(\xi_z)}{t-z}\,dt\bigg)- \frac{1}{2\pi}\,\int\limits_{\gamma_{2\rho}(\xi_z)}\big(g(t)-g(\xi_z)\big)\,d\arg(t-\xi_z)+\\
 +{\rm Re}\,\bigg(\frac{z-\xi_z}{2\pi i}\,\int\limits_{\gamma\setminus\gamma_{2\rho}(\xi_z)}\frac{g(t)-g(\xi_z)}{(t-z)(t-\xi_z)}\,dt\bigg)
=:S_1-S_2+S_3\,.
\end{multline*}

Taking into account condition (\ref{2-4:nerivnist'-teta}), we obtain the inequalities
\begin{multline*}
|S_1|\le \frac{1}{2\pi}\,\int\limits_{\gamma_{2\rho}(\xi_z)}\frac{|g(t)-g(\xi_z)|}{|t-z|}\,|dt|\le
 \frac{\omega_{\gamma}(g,2\rho)}{2\pi\rho}\,\theta_{\xi_z}(2\rho)\le c\,\omega_{\gamma}(g,2\rho)\le c\,\omega_{\gamma}(g,4\varepsilon)\le\\
 \le c\,\omega_{\gamma}(g,4\varepsilon) \,\varepsilon\,\int\limits_{4\varepsilon}^{5\varepsilon}\frac{d\eta}{\eta^2}\le
    c\,\varepsilon\,\int\limits_{4\varepsilon}^{5\varepsilon}\frac{\omega_{\gamma}(g,\eta)}{\eta^2}\,d\eta\,,
\end{multline*}
where\, $c$\, denotes different constants whose values depend only on the curve\, $\gamma$.

It is obvious that for $\rho\le\varepsilon/2$, the following inequality holds:
\[ |S_2| \le \frac{1}{2\pi}\, M_{\gamma}(g,\varepsilon)\,,\]
and for $\varepsilon/2<\rho\le 2\varepsilon$, the modulus of integral $S_2$ is estimated in Lemma \ref{lem-1}.

Just as when estimating $|I_3|$ in the proof of Lemma \ref{lem-2}, we first obtain
\[|S_3|\le c\,\rho\int\limits_{\rho}^{2d}\frac{\omega_{\gamma}(g,\eta)}{\eta^2}\,d\eta\,,\]
where the constant\, $c$\, depends only on the curve\, $\gamma$.

Further, we have the relation
\begin{multline*}
\rho\int\limits_{\rho}^{2d}\frac{\omega_{\gamma}(g,\eta)}{\eta^2}\,d\eta=
\rho\int\limits_{\rho}^{2\varepsilon}\frac{\omega_{\gamma}(g,\eta)}{\eta^2}\,d\eta+
\rho\int\limits_{2\varepsilon}^{2d}\frac{\omega_{\gamma}(g,\eta)}{\eta^2}\,d\eta \le \omega_{\gamma}(g,2\varepsilon)+2\,\varepsilon\int\limits_{2\varepsilon}^{2d}\frac{\omega_{\gamma}(g,\eta)}{\eta^2}\,d\eta=\\
=6\,\omega_{\gamma}(g,2\varepsilon) \,\varepsilon\,\int\limits_{2\varepsilon}^{3\varepsilon}\frac{d\eta}{\eta^2}+ 2\,\varepsilon\int\limits_{2\varepsilon}^{2d}\frac{\omega_{\gamma}(g,\eta)}{\eta^2}\,d\eta\le
8\,\varepsilon\int\limits_{2\varepsilon}^{2d}\frac{\omega_{\gamma}(g,\eta)}{\eta^2}\,d\eta\,,
\end{multline*}
\hfill $\Box$

\vskip 1mm

\begin{theorem}\label{theor-oz-cr-Al}
{\em Let a closed Jordan curve $\gamma$ be Ahlfors-regular and let a function
$g \colon \gamma\rightarrow \mathbb{R}$ be continuous on $\gamma$.
If condition \eqref{nidu-nepr} is satisfied, then the following estimate is true:
 \begin{equation}\label{Zygmund-log-pot}
\omega\,_{\overline{D^{\pm}}}\Big(\big({\rm Re}\,\widetilde{g}\big)^{\pm},\varepsilon\Big)\le\, c\,\left(\,M_{\gamma}(g,\varepsilon)+ \varepsilon\,\int\limits_{\varepsilon}^{2d}\frac{\omega_{\gamma}(g,\eta)}{\eta^2}\,d\eta \right) \qquad \forall\,\varepsilon\in(0,d/8]\,,
\end{equation}
where the constant\, $c$\, depends only on\, $\gamma$\, but does not depend on $\varepsilon$. }
\end{theorem}

\vskip 1mm

\textit{\textbf{Proof.}}
Let us prove estimate (\ref{Zygmund-log-pot}) for $\omega\,_{\overline{D^{+}}}\Big(\big({\rm Re}\,\widetilde{g}\big)^{+},\varepsilon\Big)$.

Let\, $0<\varepsilon\le d/8$\,.

Consider four cases of arrangement of the points $z_1, z_2\in \overline{D^+}$\,.
	
	{\em Case 1.} Let $z_1, z_2\in\gamma$ and $|z_1-z_2|\le\varepsilon$. In this case, the estimate proved in Lemma \ref{lem-2} is valid with
$\xi_1=z_1$ and $\xi_2=z_2$.

	{\em Case 2.} Let $z_1\in\gamma$, $z_2\in D^+$ and $|z_1-z_2|\le\varepsilon$.
Let $\xi_{z_2}$ be one of the closest points of the curve $\gamma$ to the point $z_2$. Then we have the equality
\[\big({\rm Re}\,\widetilde{g}\big)^{+}(z_2)-\big({\rm Re}\,\widetilde{g}\big)^{+}(z_1)=
\big({\rm Re}\,\widetilde{g}\big)^{+}(z_2)-\big({\rm Re}\,\widetilde{g}\big)^{+}(\xi_{z_2})+\big({\rm Re}\,\widetilde{g}\big)^{+}(\xi_{z_2})-\big({\rm Re}\,\widetilde{g}\big)^{+}(z_1)\,.\]

Inasmuch as $|z_2-\xi_{z_2}|\le\varepsilon$ and $|z_1-\xi_{z_2}|\le 2\varepsilon$, then $\Big|\big({\rm Re}\,\widetilde{g}\big)^{+}(z_2)-\big({\rm Re}\,\widetilde{g}\big)^{+}(\xi_{z_2})\Big|$ is estimated in Lemma~\ref{lem-3} and $\Big|\big({\rm Re}\,\widetilde{g}\big)^{+}(\xi_{z_2})-\big({\rm Re}\,\widetilde{g}\big)^{+}(z_1)\Big|$ is estimated in Lemma \ref{lem-2}.
As a result, we obtain an estimate of form
(\ref{Zygmund-log-pot-boundary}) with $\xi_1=z_1$ and $\xi_2=z_2$.

	{\em Case 3.} Let $z_1, z_2\in D^+$, $|z_1-z_2|\le\varepsilon$, $\min\limits_{t\in\gamma}|t-z_1|\le 2\varepsilon$ and $\min\limits_{t\in\gamma}|t-z_2|\le 2\varepsilon$.
Let $\xi_{z_k}$ be one of the closest points of the curve $\gamma$ to the point $z_k$ for $k=1,2$. Then we have the equality
\begin{multline*}
\big({\rm Re}\,\widetilde{g}\big)^{+}(z_1)-\big({\rm Re}\,\widetilde{g}\big)^{+}(z_2)=
\big({\rm Re}\,\widetilde{g}\big)^{+}(z_1)-\big({\rm Re}\,\widetilde{g}\big)^{+}(\xi_{z_1})+\\[2mm]
+\big({\rm Re}\,\widetilde{g}\big)^{+}(\xi_{z_1})-\big({\rm Re}\,\widetilde{g}\big)^{+}(\xi_{z_2})
+\big({\rm Re}\,\widetilde{g}\big)^{+}(\xi_{z_2})-\big({\rm Re}\,\widetilde{g}\big)^{+}(z_2)\,.
\end{multline*}

Inasmuch as $|z_1-\xi_{z_1}|\le 2\varepsilon$, $|z_2-\xi_{z_2}|\le 2\varepsilon$ and $|\xi_{z_1}-\xi_{z_2}|\le 5\varepsilon$, then
$\Big|\big({\rm Re}\,\widetilde{g}\big)^{+}(z_1)-\big({\rm Re}\,\widetilde{g}\big)^{+}(\xi_{z_1})\Big|$ and
$\Big|\big({\rm Re}\,\widetilde{g}\big)^{+}(z_2)-\big({\rm Re}\,\widetilde{g}\big)^{+}(\xi_{z_2})\Big|$ is estimated in Lemma~\ref{lem-3}, and $\Big|\big({\rm Re}\,\widetilde{g}\big)^{+}(\xi_{z_1})-\big({\rm Re}\,\widetilde{g}\big)^{+}(\xi_{z_2})\Big|$ is estimated in Lemma \ref{lem-2}. As a result, we obtain an estimate of form (\ref{Zygmund-log-pot-boundary}) with $\xi_1=z_1$ and $\xi_2=z_2$.

	{\em Case 4.} Let $z_1, z_2\in D^+$, $|z_1-z_2|\le\varepsilon$ and $\min\limits_{t\in\gamma}|t-z_2|> 2\varepsilon$. Then
$\min\limits_{t\in\gamma}|t-z_1|> \varepsilon$.

Let $\xi_{z_1}$ be one of the closest points of the curve $\gamma$ to the point $z_1$.
We use the following representation of the difference:
\begin{multline*}
\big({\rm Re}\,\widetilde{g}\big)^{+}(z_1)-\big({\rm Re}\,\widetilde{g}\big)^{+}(z_2)=
{\rm Re}\,\bigg(\frac{z_1-z_2}{2\pi i}\,
\int\limits_{\gamma}\frac{g(t)-g(\xi_{z_1})}{(t-z_1)(t-z_2)}\,dt\bigg)=\\
={\rm Re}\,\bigg(\frac{z_1-z_2}{2\pi i}\,
\int\limits_{\gamma_{\varepsilon}(\xi_{z_1})}\frac{g(t)-g(\xi_{z_1})}{(t-z_1)(t-z_2)}\,dt\bigg)+
{\rm Re}\,\bigg(\frac{z_1-z_2}{2\pi i}\,
\int\limits_{\gamma\setminus\gamma_{\varepsilon}(\xi_{z_1})}\frac{g(t)-g(\xi_{z_1})}{(t-z_1)(t-z_2)}\,dt\bigg)=:s_1+s_2.
\end{multline*}

Similarly to the estimate of $|S_1|$ in the proof of Lemma \ref{lem-3}, we obtain the estimate
\[|s_1|\le c\,\varepsilon\int\limits_{\varepsilon}^{2\varepsilon}\frac{\omega_{\gamma}(g,\eta)}{\eta^2}\,d\eta\,,\]
where the constant\, $c$\, depends only on the curve\, $\gamma$.

In addition, taking into account the inequalities $|t-z_1|\ge |t-\xi_{z_1}|/2$ and $|t-z_2|\ge |t-\xi_{z_1}|/3$,
which are fulfilled in this case for all $t\in\gamma$,
similarly to the estimate of $|I_3|$ in the proof of Lemma \ref{lem-2}, we obtain the estimate
\[|s_2|\le c\,\varepsilon\int\limits_{\varepsilon}^{2d}\frac{\omega_{\gamma}(g,\eta)}{\eta^2}\,d\eta\,,\]
where the constant\, $c$\, depends only on the curve\, $\gamma$.

An obvious corollary of the mentioned estimates for  $|s_1|$ and $|s_2|$ is the inequality
\[ \Big|\big({\rm Re}\,\widetilde{g}\big)^{+}(z_1)-\big({\rm Re}\,\widetilde{g}\big)^{+}(z_2)\Big|\le
 c\,\varepsilon\int\limits_{\varepsilon}^{2d}\frac{\omega_{\gamma}(g,\eta)}{\eta^2}\,d\eta\,,\]
where the constant\, $c$\, depends only on the curve\, $\gamma$.

Since all possible cases of arrangement of the points $z_1, z_2\in \overline{D^+}$ are considered,
the proof of estimate (\ref{Zygmund-log-pot}) for $\omega\,_{\overline{D^{+}}}\Big(\big({\rm Re}\,\widetilde{g}\big)^{+},\varepsilon\Big)$
is complete.
Estimate (\ref{Zygmund-log-pot}) for $\omega\,_{\overline{D^{-}}}\Big(\big({\rm Re}\,\widetilde{g}\big)^{-},\varepsilon\Big)$ is similarly proved.
\hfill $\Box$

\vskip 2mm

Let us note that if the modulus of continuity of the function\, $g : \gamma\rightarrow\mathbb{R}$\,,
which is given on a closed Ahlfors-regular Jordan curve,
satisfies the Dini condition
 \begin{equation}\label{2-4:Dini}
\int\limits_0^1\frac{\omega_{\gamma}(g,\eta)}{\eta}\,d\eta<\infty,
\end{equation}
then Cauchy-type integral (\ref{C-type-int}) has the limiting values $\widetilde{g}\,^{\pm}(\xi)$\, at every point\, $\xi\in\gamma$\, from the domain\, $D^{\pm}$, and the following Zygmund estimate holds:
 \begin{equation}\label{Zygmund-2}
\omega_{\overline{D^{\pm}}}(\widetilde g,\varepsilon)\le\, c\,\left(\,\int\limits_{0}^{\varepsilon}\frac{\omega_{\gamma}(g,\eta)}{\eta}\,d\eta+ \varepsilon\,\int\limits_{\varepsilon}^{2d}\frac{\omega_{\gamma}(g,\eta)}{\eta^2}\,d\eta \right)\,,
\end{equation}
where the constant\, $c$\, depends only on\, $\gamma$\, but does not depend on $\varepsilon$
(see O.F.~Gerus \cite{G77}, and also V.V.~Salaev \cite{Salaev}, where the Dini
condition of form (\ref{2-4:Dini}) is given in terms of the regularized modulus of continuity using
the Stechkin construction).

Condition (\ref{2-4:Dini}) is sufficient for condition \eqref{nidu-nepr} to be satisfied in the case where $\gamma$ is an Ahlfors-regular curve, but it is not necessary for this. Indeed, if condition \eqref{nidu-nepr} is satisfied, but the condition
\[\sup\limits_{\xi\in\gamma}\,\sup\limits_{\delta\in(0,\varepsilon)}\,\bigg|\,\int\limits_{\gamma_{\varepsilon}(\xi)\setminus\gamma_{\delta}(\xi)}
 \frac{g(t)-g(\xi)}{|t-\xi|}\,d|t-\xi|\,\bigg|\to 0\,, \qquad \varepsilon\to 0\,,\]
is not satisfied, then condition (\ref{2-4:Dini}) is also not satisfied, and the function ${\rm Im}\,\widetilde{g}(z)$ have no continuous extension to the boundary $\gamma$ from the domains $D^+$ and $D^-$, while the function ${\rm Re}\,\widetilde{g}(z)$ is continuously extended to $\gamma$ from $D^+$ and $D^-$. Thus, estimate (\ref{Zygmund-log-pot}) is an improvement of estimate (\ref{Zygmund-2}) for the logarithmic double layer potential in the case where an integration curve is Ahlfors-regular.

It will be proved further that the estimate (\ref{Zygmund-log-pot}) is exact with respect to the order of smallness as $\varepsilon\to 0$.

\vskip 2mm

\section{Estimates for the modulus of continuity of the real part of the Cauchy-type integral in the case of a Kr\'al integration curve}
Note that in the case of an Ahlfors-regular curve, for each $\xi\in\gamma$ and each $\delta>0$, the function $\arg(t-\xi)$ has a bounded variation on the set $\gamma\setminus\gamma_{\delta}(\xi)$. However, in general, the function $\arg(t-\xi)$ can be a function of unbounded variation on $\gamma$, because, in particular, it can be unbounded in a neighborhood of the point $\xi$.

Consider the class of curves  $\gamma$, for which the function $\arg(t-\xi)$ has a bounded variation  $V_{\gamma}[\arg(t-\xi)]$ on $\gamma\setminus\{\xi\}$ for all  $\xi\in\gamma$ and, moreover, satisfies the condition
\begin{equation}\label{obm-var-arg}
\sup\limits_{\xi\in\gamma} V_{\gamma}[\arg(t-\xi)]<\infty\,.
\end{equation}
Condition \eqref{obm-var-arg} is equivalent to condition \eqref{um-Kral}, which follows from the Banach indicatrix theorem (see J.~Kr\'al \cite[Lemma 1.2]{Kral-1-1964}). Therefore, curves satisfying condition \eqref{obm-var-arg} will be called the {\it Kr\'al curves}.

In the case where the integration curve is a Kr\'al curve, limiting values \eqref{gran-znach-double-pot} exist everywhere on $\gamma$ for an arbitrary continuous function $g \colon \gamma\rightarrow \mathbb{R}$ (see J.~Kr\'al \cite{Kral-1-1964}). In this case,
estimate (\ref{Zygmund-log-pot}) for the modulus of continuity of the limiting values of the real part of the Cauchy-type integral is specified in the following theorem.

\vskip 1mm

\begin{theorem}\label{theor-oz-cr-Kral}
{\em Let a closed Jordan curve $\gamma$ be a Kr\'al curve and let a function
$g \colon \gamma\rightarrow \mathbb{R}$ be continuous on $\gamma$.
Then the following estimate is true:
 \begin{equation}\label{Zygmund-log-pot-cr-Kral}
\omega\,_{\overline{D^{\pm}}}\Big(\big({\rm Re}\,\widetilde{g}\big)^{\pm},\varepsilon\Big)\le\, c\, \varepsilon\,\int\limits_{\varepsilon}^{2d}\frac{\omega_{\gamma}(g,\eta)}{\eta^2}\,d\eta \qquad \forall\,\varepsilon\in(0,d/8]\,,
\end{equation}
where the constant\, $c$\, depends only on\, $\gamma$\, but does not depend on $\varepsilon$. }
\end{theorem}

\vskip 1mm

\textit{\textbf{Proof.}}
Under condition (\ref{obm-var-arg}) for the curve $\gamma$, the following inequalities hold:
\begin{multline}\label{oz-M-Kral-kr}
M_{\gamma}(g,\varepsilon)\le  \sup\limits_{\xi\in\gamma}\,\sup\limits_{\delta\in(0,\varepsilon)}\,\int\limits_{\gamma_{\varepsilon}(\xi)\setminus\gamma_{\delta}(\xi)}
 \big|g(t)-g(\xi)\big|\,\big|d\arg(t-\xi)\big|\le \\
 \le \omega_{\gamma}(g,\varepsilon)\,\sup\limits_{\xi\in\gamma} V_{\gamma}[\arg(t-\xi)]
 \le c\, \varepsilon\,\int\limits_{\varepsilon}^{2\varepsilon}\frac{\omega_{\gamma}(g,\eta)}{\eta^2}\,d\eta\,,
 \end{multline}
where the constant\, $c$\, depends only on\, $\gamma$, and therefore estimate (\ref{Zygmund-log-pot}) takes form
(\ref{Zygmund-log-pot-cr-Kral}).
\hfill $\Box$

\vskip 2mm

Let us note that the upper estimate \eqref{Zygmund-log-pot-cr-Kral} for the solid modulus of continuity  $\omega\,_{\overline{D^{\pm}}}\Big(\big({\rm Re}\,\widetilde{g}\big)^{\pm},\varepsilon\Big)$, which is considered throughout the entire class of Kr\'al curves $\gamma$ and the entire class of continuous functions $g \colon \gamma\rightarrow\mathbb{R}$, is exact with respect to the order of smallness as $\varepsilon\to 0$.
The specified accuracy of estimate \eqref{Zygmund-log-pot-cr-Kral} will be proved
by constructing an example of a curve $\gamma$ and a continuous function $g \colon \gamma\rightarrow~\mathbb{R}$
such that for the solid modulus of continuity  $\omega\,_{\overline{D^{\pm}}}\Big(\big({\rm Re}\,\widetilde{g}\big)^{\pm},\varepsilon\Big)$,
the lower estimate is realized by the right-hand side of inequality \eqref{Zygmund-log-pot-cr-Kral}.

For fixed constants\, $\sigma\ge 1$\, and\, $k\ge 0$\,, P.M.~Tamrazov \cite[p.~58]{tamrazov} introduced the notion of a {\em normal majorant of the class $(\sigma,k)$} as a non-decreasing function $\mu(\eta)$ satisfying the inequality
\[\mu(\lambda\eta)\le\, \sigma\, \lambda^k \mu(\eta)  \qquad \forall\,\lambda>1 \quad \forall\,\eta>0\,.\]

In the case where $\mu(\eta)$ is a normal majorant of the class\, $(\sigma,1)$\,, the following inequality holds:
\begin{equation}\label{mu-nasl-pivad}
\frac{\mu(\eta_2)}{\eta_2}\le \sigma\,\frac{\mu(\eta_1)}{\eta_1} \qquad\forall\,\eta_2>\eta_1>0\,.
\end{equation}

Note that in the case of the circle $\gamma$ it follows
from the properties of the modulus of continuity (see, for example, V.K.~Dzyadyk and I.A.~Shevchuk \cite{Dz-Shev-6-3}) that the modulus of continuity $\omega_{\gamma}(g,\eta)$ of an arbitrary continuous function $g \colon \gamma\rightarrow \mathbb{R}$ is a normal majorant of the class\, $(2,1)$.

\vskip 2mm

The following statement is true.

\vskip 1mm

\begin{theorem}\label{theor-til-oz-znyzu}
{\em Let $\gamma=\{t\in\mathbb{C} : |t+1|=1\}$
and let $\mu(\eta)$ be a normal majorant of the class\, $(\sigma,1)$ such that $\mu(0^+)=0$.
Then for a continuous function $g \colon \gamma\rightarrow \mathbb{R}$ given by the equality
\[ g(t)=
   \displaystyle 8\, \int\limits_0^1 \mu(\eta)
   \,\frac{\big(({\rm Re}\,t)^2-({\rm Im}\,t)^2\big)
   (\eta^4+|t|^4)-2\eta^2|t|^4}{\big|\eta^2-t^2\big|^4}\,\eta\,d\eta \qquad \forall\,t\in\gamma\,, \]
the following inequality holds:
\begin{equation}
\label{est-up-f}
\omega_{\gamma}(g,\varepsilon) \le \, C\,\mu(\varepsilon) \qquad \forall\,\varepsilon\ge 0\,,
\end{equation}
where the constant\, $C$\, does not depend on $\varepsilon$,
and for the solid modulus of continuity of the function
$\big({\rm Re}\,\widetilde{g}\big)^{\pm}$ 
 the following estimate is true:
 \begin{equation}\label{log-pot-oz-znyzu+}
\omega\,_{\overline{D^{\pm}}}\Big(\big({\rm Re}\,\widetilde{g}\big)^{\pm},\varepsilon\Big)\ge\, c\, \varepsilon\,\int\limits_{\varepsilon}^{4}\frac{\mu(\eta)}{\eta^2}\,d\eta \qquad \forall\,\varepsilon\in(0,1/4]\,,
\end{equation}
where a real constant\, $c$\, does not depend on $\varepsilon$. }
\end{theorem}

\vskip 2mm

Theorem \ref{theor-til-oz-znyzu} and inequalities \eqref{est-up-f} and \eqref{oz-M-Kral-kr} imply the next statement.

\vskip 1mm

\begin{corolary}\label{cor-til-oz-znyzu}
{\em  Under the conditions of Theorem \ref{theor-til-oz-znyzu}, for all $\varepsilon\in(0,1/4]$, the following estimates are true:
\begin{equation}\label{tochn-ozinok}
\omega\,_{\overline{D^{\pm}}}\Big(\big({\rm Re}\,\widetilde{g}\big)^{\pm},\varepsilon\Big)\ge\, c\, \varepsilon\,\int\limits_{\varepsilon}^{4}\frac{\omega_{\gamma}(g,\eta)}{\eta^2}\,d\eta
\ge \, c_1\,\left(\,M_{\gamma}(g,\varepsilon)+ \varepsilon\,\int\limits_{\varepsilon}^{4}\frac{\omega_{\gamma}(g,\eta)}{\eta^2}\,d\eta \right)\,,
\end{equation}
where positive constants\, $c$\, and\, $c_1$\,  do not depend on $\varepsilon$. }
\end{corolary}

\vskip 1mm

Thus, inequalities \eqref{tochn-ozinok} show that the upper estimates for the solid moduli of continuity
of the logarithmic double layer potential, which are given in Theorems \ref{theor-oz-cr-Al} and \ref{theor-oz-cr-Kral}, are exact with respect to the order of smallness as $\varepsilon\to 0$.

The authors dedicate their work to the memory of Professor Oleg Gerus.

\vskip 2mm

{\bf Author Contributions.}
Special cases of Theorems 1 and 2, including
estimates for the moduli of continuity on the integration curve, were obtained jointly by the authors.
All estimates for the solid moduli of continuity were obtained by the first author.

{\bf Acknowledgements.}
The first author was supported by the National
Research Foundation of Ukraine, Project number 2025.07/0014, Project name "Modern problems of Mathematical Analysis and Geometric Function Theory"\/.




\end{document}